\documentclass[11pt,letterpaper,reqno]{amsart}
\usepackage{tikz}
\usetikzlibrary{positioning, shapes.geometric, arrows.meta, calc, backgrounds, fit}
\usepackage{amssymb}
\usepackage{amsmath}
\usepackage{amsthm}
\usepackage{amsfonts}
\usepackage{bbm}
\usepackage{enumitem} 
\usepackage{pgfplots}
\pgfplotsset{compat=1.18} 
\usepackage{booktabs}
\usepackage{graphicx}
\usepackage[T1]{fontenc}
\usepackage{hyperref}
\usepackage{bookmark}
\usepackage{doi}
\usepackage{float} 
\usepackage{comment} 
\hypersetup{pdfstartview={FitH}}

\newtheorem{theorem}{Theorem}[section]
\newtheorem{lemma}[theorem]{Lemma}
\newtheorem{proposition}[theorem]{Proposition}

\newtheorem{problem}[theorem]{Problem}

\theoremstyle{definition}

\newtheorem{definition}[theorem]{Definition}

\let\originalleft\left
\let\originalright\right
\renewcommand{\left}{\mathopen{}\mathclose\bgroup\originalleft}
\renewcommand{\right}{\aftergroup\egroup\originalright}

\begin{document}

\title{Largest Sidon subsets in weak Sidon sets}

\author[J.~Ma]{Jie Ma}
\author[Q.~Tang]{Quanyu Tang}

\address{School of Mathematical Sciences, University of Science and Technology of China, Hefei, Anhui 230026, and Yau Mathematical Sciences Center, Tsinghua University, Beijing 100084, China}
\email{jiema@ustc.edu.cn}
\address{School of Mathematics and Statistics, Xi'an Jiaotong University, Xi'an 710049, P. R. China}
\email{tang\_quanyu@163.com}




\begin{abstract}
A finite set $ S \subset \mathbb{R} $ is called a Sidon set if all sums $ x+y $ with $ x,y \in S $ and $ x \le y $ are distinct, and a weak Sidon set if all sums $ x+y $ with $ x,y \in S $ and $ x < y $ are distinct.  
For a finite set $ A \subset \mathbb{R} $, let $ h(A) $ denote the maximum size of a Sidon subset of $ A $, and define
$$
g(n) := \min\{\, h(A) : A \subset \mathbb{R},\ |A| = n,\ A \text{ is a weak Sidon set} \,\}.
$$
S\'ark\"ozy and S\'os asked whether the limit $ \lim_{n\to\infty} g(n)/n $ exists and, if so, to determine its value. 
We resolve this problem completely by determining $g(n)$ exactly:
$$
g(n)=\left\lceil \frac{n+1}{2}\right\rceil
\qquad\text{for all } n\ge 1.
$$
In particular, $\lim_{n\to\infty} g(n)/n=\frac12$.

We also investigate a related problem of Erd\H{o}s concerning a local difference condition. 
A finite set $ A \subset \mathbb{R} $ is called a $(4,5)$-set if every $4$-element subset of $A$ determines at least five distinct values among its six pairwise absolute differences. 
Erd\H{o}s asked for the optimal constant $ c_* > 0 $ such that every $(4,5)$-set of size $ n $ contains a Sidon subset of size at least $ c_* n $. 
Gy\'arf\'as and Lehel reduced this to an extremal problem of $3$-uniform hypergraphs and proved $\frac{1}{2} + \frac{1}{141 \cdot 76} \le c_* \le \frac{3}{5}$. 
We improve both bounds by establishing
$$
\frac{9}{17} \le c_* \le \frac{4}{7},
$$
where the lower bound uses a reformulation of the extremal problem, and the upper bound follows from an explicit construction together with a convenient characterization of $c_*$.
\end{abstract}

\maketitle

\section{Introduction}\label{sec:Introduction}

Throughout this paper, a finite set \( S \subset \mathbb{R} \) is called a \emph{Sidon} set if all pairwise sums \( x+y \) with \( x,y \in S \) and \( x \le y \) are distinct. 
For a finite set \( A \subset \mathbb{R} \), let \( h(A) \) denote the maximum cardinality of a Sidon subset of \( A \).

Sidon sets are among the most intensively studied objects in additive combinatorics, owing to their deep connections with extremal combinatorics, number theory, and discrete geometry; see, e.g.,~\cite{Obr04}.
Motivated by problems posed by Erd\H{o}s and by S\'ark\"ozy and S\'os, we study the asymptotic density of the largest Sidon subsets contained in sets that are close to being Sidon.
We consider this problem under two related but distinct notions of proximity to the Sidon property, to be elaborated in the following two subsections.

\subsection{\texorpdfstring{A problem of S\'ark\"ozy and S\'os on weak Sidon sets}{A problem of Sarkozy and Sos on weak Sidon sets}}\label{subsec:SaSo-problem}
The first notion of proximity to being Sidon that we consider arises naturally. 
A finite set \( A \subset \mathbb{R} \) is called a \emph{weak Sidon} set if all sums \( x+y \) with \( x,y \in A \) and \( x<y \) are distinct; see, e.g.,~\cite{Kay05}.
In~\cite[p.~247]{SaSo13}, S\'ark\"ozy and S\'os asked how large a Sidon subset must be contained in a weak Sidon set. 
To state this formally, for each positive integer \( n\), define
\[
g(n):=\min\{\,h(A):\ A\subset\mathbb{R},\ |A|=n,\ A\text{ is a weak Sidon set}\,\}.
\]
They then explicitly posed the following problem.\footnote{The authors of~\cite{SaSo13} noted that the result of Gy\'arf\'as and Lehel~\cite{GyLe95} (which concerns $(4,5)$-sets; see Section~\ref{subsec:erdos-45}) may appear relevant to this problem. However, \cite{GyLe95} does not directly yield progress on it, since a weak Sidon set need not be a $(4,5)$-set; see Section~\ref{subsec:relations}.}

\begin{problem}[S\'ark\"ozy and S\'os \cite{SaSo13}, Problem~12]\label{prob:SaSo13_1}
Prove that the limit $\lim_{n\to\infty} \frac{g(n)}{n}$ exists and determine the limit.
\end{problem}

Our main result of this paper resolves Problem~\ref{prob:SaSo13_1} by proving the existence of the limit and determining the corresponding asymptotic constant.

\begin{theorem}\label{thm:limit_exists_v2}
The limit \( \gamma_* := \lim_{n\to\infty} \frac{g(n)}{n} \) exists and satisfies \( \gamma_*=\frac12 \).
\end{theorem}
In fact, we prove the following stronger statement, which determines $g(n)$ exactly.

\begin{theorem}\label{thm:gamma_star_bounds}
For every positive integer $n$,
\[
g(n)=\left\lceil \frac{n+1}{2}\right\rceil.
\]
\end{theorem}

We refer readers to Section~\ref{subsec:proof-ideas} for the proof sketches of these results.

\subsection{\texorpdfstring{A problem of Erd\H{o}s on $(4,5)$-sets}{A problem of Erdos on (4,5)-sets}}\label{subsec:erdos-45}

Our second notion of proximity to being Sidon, first studied by Erd\H{o}s~\cite{Er97b}, is defined as follows.
A finite set \(A\subset\mathbb{R}\) is called a \emph{\((4,5)\)-set} if for every choice of four distinct elements
\(x_1,x_2,x_3,x_4\in A\), one has\footnote{The notions of Sidon sets, \((4,5)\)-sets, and weak Sidon sets form a strict chain of implications; see Figure~\ref{fig:three-notions} and the discussion in Section~\ref{subsec:relations}.}
\[
\Bigl|\bigl\{|x_i-x_j|:\ 1\le i<j\le 4\bigr\}\Bigr|\ge 5.
\]
This concept is closely related to an anti-Ramsey problem (see~\cite{BCDP24,JM24}). 
In~\cite[p.~231]{Er97b}, Erd\H{o}s posed the following problem, which is also listed as Problem~\#757 on Bloom's website~\cite{EP757}.

\begin{problem}[Erd\H{o}s \cite{Er97b}]\label{prob:erdos-757}
Let $A\subset \mathbb{R}$ be a $(4,5)$-set of size $n$. Find the best constant $c_*>0$ such that $A$ must always contain a Sidon set of size at least $c_*n$.
\end{problem}

\noindent For convenience, for each positive integer \( n \), define
\[
f(n)\ :=\ \min\{\,h(A):\ A\subset \mathbb{R},\ |A|=n,\ A\text{ is a }(4,5)\text{-set}\,\}.
\]
Gy\'arf\'as and Lehel~\cite{GyLe95} proved that for all positive integers \(n\),
\begin{equation}\label{eq:1211417635}
\left(\frac{1}{2}+\frac{1}{141\cdot 76}\right)n  \le f(n) \le \frac{3}{5}n+1.
\end{equation}
Their upper bound is constructed using an infinite sequence based on Fibonacci numbers, while the lower bound is reduced to a problem on the transversal number of certain \(3\)-uniform linear hypergraphs.

In analogy with Theorem~\ref{thm:limit_exists_v2}, we first prove that the limit \( \lim_{n\to\infty} \frac{f(n)}{n} \) exists and equals the optimal constant \( c_* \) in Problem~\ref{prob:erdos-757}.

\begin{theorem}\label{thm:limit_exists_v1}
The limit \( \lim_{n\to\infty} \frac{f(n)}{n} \) exists and satisfies
\(
c_*=\lim_{n\to\infty}\frac{f(n)}{n}
=
\inf_{n\ge 1}\frac{f(n)}{n}.
\)
\end{theorem}
 
The bounds \eqref{eq:1211417635} of Gy\'arf\'as and Lehel~\cite{GyLe95} then correspond to
\( \frac{1}{2}+\frac{1}{141\cdot 76}  \le  c_{*}  \le  \frac{3}{5}.\)
We improve both bounds as follows.

\begin{theorem}\label{thm:c_star_bounds}
It holds that
\[
\frac{9}{17} \le  c_*  \le  \frac{4}{7}.
\]
\end{theorem}

We point out that Theorem~\ref{thm:limit_exists_v1} is essential for the upper bound in Theorem~\ref{thm:c_star_bounds}. Once
\(
c_*=\inf_{n\ge1} f(n)/n
\)
is established, any single finite \((4,5)\)-set \(A\) with a small enough ratio between its largest Sidon subset and \(|A|\) immediately yields an upper bound on \(c_*\).

\subsection{Proof strategy}\label{subsec:proof-ideas}
We outline the main ideas behind the proofs of our results in this subsection.
The proofs of the existence statement in Theorem~\ref{thm:limit_exists_v2} and Theorem~\ref{thm:limit_exists_v1} are both based on a subadditivity argument: we show that $g(n)$ and $f(n)$ are subadditive, and the existence of the limits then follows from Fekete's lemma.

We first sketch the proof strategy for Theorem~\ref{thm:c_star_bounds}. 
For the upper bound $c_*\le 4/7$, we construct an explicit $14$-point $(4,5)$-set whose largest Sidon subset has size $8$, verified by exhaustive enumeration. Combined with the identity $c_*=\inf_{n\ge 1} f(n)/n$ from Theorem~\ref{thm:limit_exists_v1}, this already yields \(c_*\le 4/7\).
For the lower bound $c_*\ge 9/17$, we build on the A.P.-hypergraph framework of Gy\'arf\'as and Lehel~\cite{GyLe95}, which we now define below.

\begin{definition}[A.P.-hypergraph]\label{defn:aphyper_v1}
Let \(A=\{a_1,\dots,a_n\}\subset \mathbb R\) be a finite set, indexed so that \(a_1<\cdots<a_n\). The \emph{A.P.-hypergraph} of \(A\) is the \(3\)-uniform hypergraph \(H(A)=(V,E)\) on vertex set \(A\) whose edges are precisely the triples
\( \{a_i,a_j,a_k\} \)
for which \(a_i,a_j,a_k\) form a \(3\)-term arithmetic progression, where \( 1\le i<j<k\le n \).
\end{definition}

Let \(\alpha(H)\) and \(\tau(H)\) denote the independence number and the transversal number of a hypergraph \(H\), respectively.
When $A$ is a $(4,5)$-set, Sidon subsets of $A$ correspond exactly to independent sets of $H(A)$; in particular,
$h(A)=\alpha(H(A))$ (see Lemma~\ref{lem:sidon-iff-apfree}).
Since $\alpha(H(A))+\tau(H(A))=|V(H(A))|=n$,
it is equivalent to working with the transversal number $\tau(H(A))$.

\begin{figure}[htbp]
\centering

\begin{tikzpicture}[scale=0.6]

\tikzstyle{vertexX}=[circle,draw, fill=black!100, minimum size=7pt, inner sep=0.2pt]

\newcommand{\hyperedgetwo}[6]{
  \pgfmathsetmacro\Done{sqrt((#4-#1)^2+(#5-#2)^2)}
  \pgfmathsetmacro\angleone{(#2>#5)*(180+asin((#4-#1)/ \Done)-asin((#1-#4)/ \Done))+asin((#1-#4)/ \Done)+asin((#3-#6)/\Done)}
  \pgfmathsetmacro\angleone{\angleone-360*(\angleone>0)-360*(\angleone>360)}
  \draw ([shift=(\angleone:#3)] #1,#2)--([shift=(\angleone:#6)]#4,#5);
  \pgfmathsetmacro\Dtwo{sqrt((#1-#4)^2+(#2-#5)^2)}
  \pgfmathsetmacro\angletwo{(#5>#2)*(180+asin((#1-#4)/ \Dtwo)-asin((#4-#1)/ \Dtwo))+asin((#4-#1)/ \Dtwo)+asin((#6-#3)/\Dtwo)}
  \pgfmathsetmacro\angletwo{\angletwo-360*(\angletwo>0)-360*(\angletwo>360)}
  \draw ([shift=(\angletwo:#6)] #4,#5)--([shift=(\angletwo:#3)]#1,#2);
  \draw (#1,#2)+(\angletwo:#3) arc(\angletwo:\angleone+360*(\angleone<\angletwo):#3);
  \draw (#4,#5)+(\angleone:#6) arc(\angleone:\angletwo+360*(\angletwo<\angleone):#6);
}

\newcommand{\hyperedgethree}[9]{
  \pgfmathsetmacro\Done{sqrt((#4-#1)^2+(#5-#2)^2)}
  \pgfmathsetmacro\angleone{(#2>#5)*(180+asin((#4-#1)/ \Done)-asin((#1-#4)/ \Done))+asin((#1-#4)/ \Done)+asin((#3-#6)/\Done)}
  \pgfmathsetmacro\angleone{\angleone-360*(\angleone>0)-360*(\angleone>360)}
  \draw ([shift=(\angleone:#3)] #1,#2)--([shift=(\angleone:#6)]#4,#5);

  \pgfmathsetmacro\Dtwo{sqrt((#7-#4)^2+(#8-#5)^2)}
  \pgfmathsetmacro\angletwo{(#5>#8)*(180+asin((#7-#4)/ \Dtwo)-asin((#4-#7)/ \Dtwo))+asin((#4-#7)/ \Dtwo)+asin((#6-#9)/\Dtwo)}
  \pgfmathsetmacro\angletwo{\angletwo-360*(\angletwo>0)-360*(\angletwo>360)}
  \draw ([shift=(\angletwo:#6)] #4,#5)--([shift=(\angletwo:#9)]#7,#8);

  \pgfmathsetmacro\Dthree{sqrt((#1-#7)^2+(#2-#8)^2)}
  \pgfmathsetmacro\anglethree{(#8>#2)*(180+asin((#1-#7)/ \Dthree)-asin((#7-#1)/ \Dthree))+asin((#7-#1)/ \Dthree)+asin((#9-#3)/\Dthree)}
  \pgfmathsetmacro\anglethree{\anglethree-360*(\anglethree>0)-360*(\anglethree>360)}
  \draw ([shift=(\anglethree:#9)] #7,#8)--([shift=(\anglethree:#3)]#1,#2);

  \draw (#1,#2)+(\anglethree:#3) arc(\anglethree:\angleone+360*(\angleone<\anglethree):#3);
  \draw (#4,#5)+(\angleone:#6) arc(\angleone:\angletwo+360*(\angletwo<\angleone):#6);
  \draw (#7,#8)+(\angletwo:#9) arc(\angletwo:\anglethree+360*(\anglethree<\angletwo):#9);
}

\node (T)  at (0,4.2)   [vertexX] {};   
\node (ML) at (-2.1,2.1)[vertexX] {};   
\node (MC) at (0,2.1)   [vertexX] {};   
\node (MR) at (2.1,2.1) [vertexX] {};   
\node (BL) at (-4.2,0)  [vertexX] {};   
\node (BC) at (0,0)     [vertexX] {};   
\node (BR) at (4.2,0)   [vertexX] {};   

\tikzset{hedge/.style={color=black, line width=0.9pt}}

\begin{scope}[hedge]

  \hyperedgetwo{0}{4.2}{0.45}{-4.2}{0}{0.45}

  \hyperedgetwo{0}{4.2}{0.43}{0}{0}{0.43}

  \hyperedgetwo{0}{4.2}{0.45}{4.2}{0}{0.45}

  \hyperedgetwo{-2.1}{2.1}{0.45}{2.1}{2.1}{0.45}

  \hyperedgetwo{-4.2}{0}{0.45}{4.2}{0}{0.45}

\end{scope}

\end{tikzpicture}
\caption{The \(7\)-vertex \(3\)-uniform hypergraph \(F_7\).}
\label{fig:F7}

\end{figure}
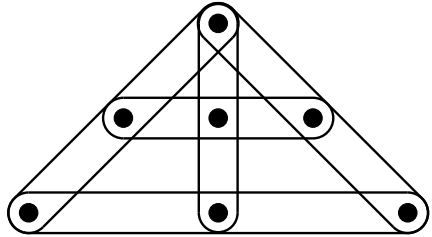

Moreover, the condition of being a \((4,5)\)-set imposes strong structural restrictions on \( H(A) \). 
One can show that the hypergraph \( H(A) \) is linear, that different edges (corresponding to \(3\)-term arithmetic progressions in \( A \)) have distinct middle points (with respect to the natural order on \( \mathbb{R} \)), and that \( H(A) \) is \( F_7 \)-free. 
Here and throughout, \( F_7 \) denotes the \(3\)-uniform hypergraph on vertex set \( \{0,1,2,3,4,5,6\} \) with
\(
E(F_7)=\bigl\{\{0,1,2\},\{0,3,4\},\{0,5,6\},\{1,3,5\},\{2,4,6\}\bigr\},
\)
as illustrated in Figure~\ref{fig:F7}.
Gy\'arf\'as and Lehel~\cite[Theorem~3.2]{GyLe95} combined this reduction with a transversal estimate to obtain the
lower bound in~\eqref{eq:1211417635}. 
Using the same framework, but replacing their transversal estimate
by the stronger inequality of Henning and Yeo~\cite[Theorem~7]{HeYe21} for \(3\)-uniform linear \(F_7\)-free
hypergraphs, we obtain the improved lower bound $c_*\ge 9/17$.

We now discuss Theorem~\ref{thm:gamma_star_bounds}. For the lower bound, let $A$ be a weak Sidon set and pass to its A.P.-hypergraph $H(A)$. Combining the general inequality \(|E(H(A))|\le |V(H(A))|-2\) (see Lemma~\ref{lem:mleqn}) with an upper bound on the transversal number of a $3$-uniform hypergraph due to Chv\'{a}tal and McDiarmid~\cite{ChMc92} and independently Tuza~\cite{Tuz90} (see Theorem~\ref{thm:CMT_v1}) gives \(h(A)\ge |A|/2+1/2\), and hence \(g(n)\ge \left\lceil \frac{n+1}{2}\right\rceil\). For the matching upper bound, for each $k\ge 2$ we construct a weak Sidon set $A_{2k+1}$ with $|A_{2k+1}|=2k+1$ and $h(A_{2k+1})=k+1$ (see Proposition~\ref{prop:A2n1_weak_sidon} and Theorem~\ref{thm:A2n1_h_equals_nplus1}), which gives \(g(2k+1)\le k+1\). For even orders, any $2k$-element subset of $A_{2k+1}$ is again weak Sidon and has largest Sidon subset of size at most $k+1$, so \(g(2k)\le k+1\). Together with the lower bound, this gives \(g(n)=\left\lceil \frac{n+1}{2}\right\rceil\) for all $n\ge 1$, where the cases $n=1,2,3$ are immediate. In particular, \(\gamma_*=\frac12\); combined with the existence of the limit established earlier, this completes the proof of Theorem~\ref{thm:limit_exists_v2}.

\subsection{Paper organization}\label{subsec:organization}

Section~\ref{sec:preliminaries} records the relationships among Sidon sets, \((4,5)\)-sets, and weak Sidon sets, and collects basic facts about A.P.-hypergraphs needed for our lower-bound arguments. Section~\ref{sec:existence_of_limit} establishes the existence of the asymptotic constants $\gamma_*$ and $c_*$ by proving the first assertion of Theorem~\ref{thm:limit_exists_v2} and Theorem~\ref{thm:limit_exists_v1}. Section~\ref{sec:weak_sidon} proves Theorem~\ref{thm:gamma_star_bounds}, namely, \(g(n)=\left\lceil \frac{n+1}{2}\right\rceil\) and hence yields the second assertion of Theorem~\ref{thm:limit_exists_v2} that $\gamma_*=\frac12$. 
Section~\ref{sec:c_star_bounds} proves Theorem~\ref{thm:c_star_bounds}, improving both bounds for \(c_*\).

\section{Preliminaries}\label{sec:preliminaries}

\subsection{\texorpdfstring{Relations between Sidon sets, $(4,5)$-sets, and weak Sidon sets}{Sidon sets, (4,5)-sets, and weak Sidon sets}}\label{subsec:relations}

We first establish a general relation between Sidon sets, \((4,5)\)-sets, and weak Sidon sets, as illustrated in Figure~\ref{fig:three-notions}.

\medskip

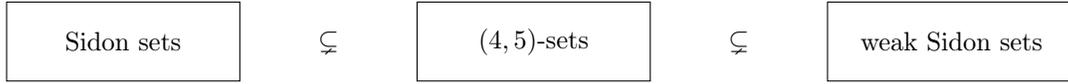
\begin{figure}[!htbp]
\centering
\begin{tikzpicture}[every node/.style={font=\small}]
  \node[draw, rectangle, minimum width=3.1cm, minimum height=1.05cm, align=center] (sidon) {Sidon sets};
  \node[right=0.9cm of sidon, font=\Large] (imp1) {$\subsetneq$};
  \node[draw, rectangle, minimum width=3.1cm, minimum height=1.05cm, align=center, right=0.9cm of imp1] (f45) {$(4,5)$-sets};
  \node[right=0.9cm of f45, font=\Large] (imp2) {$\subsetneq$};
  \node[draw, rectangle, minimum width=3.3cm, minimum height=1.05cm, align=center, right=0.9cm of imp2] (wsidon) {weak Sidon sets};
\end{tikzpicture}
\medskip
\caption{Inclusions among Sidon sets, $(4,5)$-sets, and weak Sidon sets.}
\label{fig:three-notions}
\end{figure}

We prove these implications in the following two propositions and give examples showing that the inclusions are strict.

\begin{proposition}\label{prop:sidon_implies_45}
Every Sidon set is a \((4,5)\)-set.
\end{proposition}

\begin{proof}
Let \(S\subset\mathbb R\) be a Sidon set, and let \(x_1,x_2,x_3,x_4\in S\) be four distinct elements.
We claim that the six absolute differences \(|x_i-x_j|\) \((1\le i<j\le 4)\) are all distinct.

Suppose for contradiction that \(|x_i-x_j|=|x_k-x_\ell|\) for two distinct unordered pairs \(\{i,j\}\ne\{k,\ell\}\).
Swapping indices within each pair if necessary, we may assume \(x_i>x_j\) and \(x_k>x_\ell\), so that
\(x_i-x_j=x_k-x_\ell\). Rearranging gives
\[
x_i+x_\ell=x_k+x_j.
\]
The unordered pairs \(\{x_i,x_\ell\}\) and \(\{x_k,x_j\}\) are distinct (otherwise one would have
\(\{x_i,x_j\}=\{x_k,x_\ell\}\), contradicting our choice of pairs), so this equality gives two different
representations of the same sum using elements of \(S\), contradicting that \(S\) is Sidon.
This shows that all six absolute differences of any four distinct elements in $S$ are distinct.
So $S$ is a \((4,5)\)-set.
\end{proof}

\begin{proposition}\label{prop:45_implies_weak}
Every $(4,5)$-set is a weak Sidon set.
\end{proposition}

\begin{proof}
Let $A\subset \mathbb{R}$ be a $(4,5)$-set. Suppose, for a contradiction, that $A$ is not weak Sidon.
Then there exist four distinct elements $a_1,a_2,a_3,a_4\in A$ such that
\[
a_1+a_4=a_2+a_3.
\]
By rearranging if necessary, we may assume $a_1<a_2<a_3<a_4$. Then the above equality implies \(a_4-a_3=a_2-a_1\), so among the six pairwise absolute differences of $\{a_1,a_2,a_3,a_4\}$, we have
\[
|a_4-a_3|=|a_2-a_1| \qquad \text{and}\qquad |a_4-a_2|=|a_3-a_1|.
\]
Thus at least two distinct values among the six differences are repeated, so there are at most four distinct
absolute differences. This contradicts the $(4,5)$-property.
\end{proof}

We remark that these two inclusions are strict. 
For example, the set \(\{1,2,3,5\}\) is weak Sidon (its six sums \(x+y\) with \(x<y\) are \(3,4,5,6,7,8\)), but it is not a \((4,5)\)-set since its six pairwise absolute differences take only four distinct values.
On the other hand, \(\{1,2,3,6\}\) is a \((4,5)\)-set, but it is not Sidon since \(1+3=2+2\).

\subsection{A.P.-hypergraphs}\label{subsec:ap-hypergraphs}

We use the A.P.-hypergraph framework of Gy\'arf\'as and Lehel~\cite{GyLe95}. For a finite set
\(A\subset \mathbb{R}\), let \(H(A)\) be its A.P.-hypergraph (Definition~\ref{defn:aphyper_v1}).
Throughout this paper, for a hypergraph \(H\) we write $n_H:=|V(H)|$ and $m_H:=|E(H)|$, and let \(\alpha(H)\) and \(\tau(H)\) denote the independence number and transversal number of \(H\), respectively.

The next three lemmas are stated in \cite[p.~109]{GyLe95} without proof. For completeness, we include full proofs.
Moreover, we weaken the hypotheses in two of them: in fact, Lemmas~\ref{lem:sidon-iff-apfree} and~\ref{lem:mleqn}
hold for all weak Sidon sets, not only for $(4,5)$-sets.

For the first lemma, the key observation is that, under the weak Sidon condition, the Sidon subsets are exactly those subsets that contain no $3$-term arithmetic progression.

\begin{lemma}\label{lem:sidon-iff-apfree}
Let \(A\) be a weak Sidon set, and let \(H=H(A)\) be its A.P.-hypergraph. For a subset \(S\subseteq A\), the following are equivalent:
\begin{enumerate}
\item \(S\) is Sidon;
\item \(S\) contains no \(3\)-term arithmetic progression.
\end{enumerate}
In particular,
\[
h(A)=\alpha(H).
\]
\end{lemma}

\begin{proof}
If \(S\) contains a \(3\)-term arithmetic progression \(\{x-d,x,x+d\}\) with \(d>0\), then
\[
(x-d)+(x+d)=x+x,
\]
so \(S\) is not Sidon.
Conversely, suppose that \(S\) is not Sidon. Then there exist \(u,v,x,y\in S\) such that
\(u+v=x+y\),
but the two representations are different. Reordering within each pair if necessary, we may assume
\(u\le v, x\le y\) with \( (u,v)\neq (x,y)\).

We first claim that \(\{u,v,x,y\}\) cannot consist of four distinct elements. Indeed, if it did, write them as
\(a<b<c<d\).
Since \(a\) is the smallest and \(d\) is the largest, the only way to split these four distinct numbers into two pairs with equal sum is
\(a+d=b+c\),
which contradicts the fact that \(A\) is weak Sidon.

Hence \(\{u,v,x,y\}\) has at most three distinct elements. It cannot have only two distinct elements either. Indeed, if the two values are \(a<b\), then the only possibilities with nondecreasing order are
\((a,a), (a,b)\) and \( (b,b),\)
whose sums are \(2a\), \(a+b\), and \(2b\), respectively, and these are pairwise distinct. Thus \(u+v=x+y\) would force \((u,v)=(x,y)\), contrary to our choice of the two representations.
Therefore \(\{u,v,x,y\}\) has exactly three distinct elements, say \(a<b<c\). It follows that the two representations must be
\[
a+c=b+b,
\]
so \(a,b,c\) form a \(3\)-term arithmetic progression.
We have proved that \(S\) is Sidon if and only if \(S\) contains no \(3\)-term arithmetic progression.

Finally, a subset of \(V(H)\) is independent in \(H\) if and only if it contains no edge of \(H\), i.e., no \(3\)-term arithmetic progression from \(A\). Therefore \(h(A)=\alpha(H)\).
\end{proof}

The following lemma on weak Sidon sets will be used frequently in the sequel.

\begin{lemma}\label{lem:mleqn}
Let \(A\) be a weak Sidon set with \(|A|\ge 2\), and let \(H=H(A)\).
For each edge \(e=\{x-d,x,x+d\}\in E(H)\) with \(d>0\), define its midpoint by \(\mu(e):=x\).
Then the map \(\mu:E(H)\to A\) is injective. In particular,
\[
m_H \le n_H - 2.
\]
\end{lemma}

\begin{proof}
Write \(A=\{a_1,\dots,a_n\}\) with \(a_1<\cdots<a_n\), so \(n=n_H\geq 2\).

We first show that the map \(\mu:E(H)\to A\) is injective. Suppose for a contradiction that
there exist two distinct edges \(e_1\) and \(e_2\) with the same midpoint \(x\). In that case, let
\[
e_1=\{x-d_1,x,x+d_1\},\qquad e_2=\{x-d_2,x,x+d_2\}
\]
for some distinct \(d_1,d_2>0\), and hence
\[
(x-d_1)+(x+d_1)=(x-d_2)+(x+d_2)=2x.
\]
This gives two distinct representations of the same sum using four distinct elements of \(A\), contradicting that $A$ is weak Sidon.

Hence, the map \(\mu:E(H)\to A\) is injective. Also, \(\mu(e)\notin\{a_1,a_n\}\) for every \(e\in E(H)\), since the midpoint
of a nontrivial \(3\)-term arithmetic progression lies strictly between its endpoints. Hence \(\mu(E(H))\subseteq A\setminus\{a_1,a_n\}\). Therefore
\[
m_H=|E(H)|=|\mu(E(H))|\le |A|-2=n_H-2.
\qedhere\]
\end{proof}

As we mentioned earlier, the condition of being a $(4,5)$-set imposes strong structural restrictions on \(H(A)\). We say that a hypergraph is
\emph{linear} if any two distinct edges intersect in at most one vertex.

\begin{lemma}\label{lem:linear}
If \(A\) is a \((4,5)\)-set, then its A.P.-hypergraph \(H(A)\) is linear.
\end{lemma}

\begin{proof}
Suppose, for a contradiction, that \(H(A)\) is not linear. Then there exist two distinct edges
\(e_1,e_2\) with \(|e_1\cap e_2|\ge 2\).

By Proposition~\ref{prop:45_implies_weak} and Lemma~\ref{lem:mleqn}, we conclude that \(e_1\) and \(e_2\) do not share their midpoint. Let the two common vertices be \(u<v\), and set \(d:=v-u>0\).
Any \(3\)-term arithmetic progression containing both \(u\) and \(v\) must be one of the triples
\[
\{u-d,u,v\},\qquad \left\{u,\frac{u+v}{2},v\right\},\qquad \{u,v,v+d\},
\]
according to whether the midpoint is \(u\), \(\frac{u+v}{2}\), or \(v\).
Since \(e_1\neq e_2\), their third vertices are two distinct choices among \(u-d,\frac{u+v}{2},v+d\),
so \(e_1\cup e_2\) consists of four points.

If the third vertices are \(u-d\) and \(v+d\), then the four points form a \(4\)-term arithmetic progression
\(\{u-d,u,v,v+d\}\), which determines only three distinct pairwise absolute differences.

Otherwise, one of the edges is \(\{u-d,u,v\}\) or \(\{u,v,v+d\}\) and the other is
\(\left\{u,\frac{u+v}{2},v\right\}\). In either case, among the six pairwise absolute differences we have
\[
|u-(u-d)|=|v-(v+d)|=|v-u|=d
\qquad\text{and}\qquad
\left|\frac{u+v}{2}-u\right|=\left|v-\frac{u+v}{2}\right|=\frac d2,
\]
so there are at most four distinct values. This contradicts the fact that $A$ is a $(4,5)$-set.
Hence \(H(A)\) is linear.
\end{proof}

We point out that the hypothesis in Lemma~\ref{lem:linear} cannot be weakened from \((4,5)\)-sets to weak Sidon sets. In fact, the A.P.-hypergraph of a weak Sidon set need not be linear. For example, let \(A=\{1,2,3,5\}\). Then \(A\) is a weak Sidon set. However, \(H(A)\) contains the two edges
\[
\{1,2,3\}\qquad\text{and}\qquad \{1,3,5\},
\]
and these intersect in the two vertices \(\{1,3\}\). Thus \(H(A)\) is not linear.

\section{Existence of the asymptotic constants}\label{sec:existence_of_limit}

In this section we prove the existence of the limits defining $\gamma_*$ and $c_*$
by establishing the first assertion of Theorem~\ref{thm:limit_exists_v2} and Theorem~\ref{thm:limit_exists_v1}.
We will use the standard Fekete's lemma for subadditive sequences.

\begin{lemma}[Fekete's lemma]\label{lem:fekete}
Let $(a_k)_{k\ge 1}$ be a sequence of real numbers such that $a_{m+n}\le a_m+a_n$ for all $m,n\ge 1$.
Then the limit $\lim_{k\to\infty} a_k/k$ exists and equals $\inf_{k\ge 1} a_k/k$.
\end{lemma}

\subsection{Weak Sidon sets: existence of the limit}

We first treat the weak Sidon set problem and prove the first assertion of Theorem~\ref{thm:limit_exists_v2}.
The following lemma is straightforward. 

\begin{lemma}\label{lem:wsidon_affine}
If $A\subset\mathbb{R}$ is a weak Sidon set, then $qA+t:=\{qa+t:\ a\in A\}$ is a weak Sidon set for every $q> 0$ and $t\in\mathbb{R}$.
Moreover, $h(qA+t)=h(A)$.
\end{lemma}

\begin{proof}
The map $x\mapsto qx+t$ is a bijection on~$\mathbb{R}$ that preserves all additive relations of the form $x+y=u+v$ (up to scaling and translation).
In particular, the sums of distinct pairs in \(A\) are distinct if and only if the corresponding sums in \(qA+t\) are distinct. Hence, the weak Sidon property is preserved.
The statement $h(qA+t)=h(A)$ follows similarly, since the Sidon property is also invariant under affine transformations.
\end{proof}

Now we establish the key technical lemma needed in the proof of Theorem~\ref{thm:limit_exists_v2}.

\begin{lemma}\label{lem:wsidon_concat}
Let $A,B\subset\mathbb{R}$ be two finite weak Sidon sets. Then there exist $q>0$ and $t\in\mathbb{R}$
such that the set
\[
C := A\cup(qB+t)
\]
is weak Sidon, satisfies $|C|=|A|+|B|$, and \(h(C) \le h(A)+h(B)\).
\end{lemma}

\begin{proof}
By Lemma~\ref{lem:wsidon_affine}, we may translate \(A\) and \(B\) so that \(\min A=\min B=0\).
Write \(D_A:=\max A\) and \(D_B:=\max B\).
If \(|A|\le 1\) or \(|B|\le 1\), say \(|A|\le 1\), fix any \(q>0\) and choose \(t\) such that \(D_A<t\) and \(t> D_A+qD_B\).
Then \(A\) lies strictly to the left of \(qB+t\), all mixed sums of the form \(t+(a+qb)\) with \(a\in A\) and \(b\in B\) are pairwise distinct, and each of them is strictly smaller than any sum within \(qB+t\).
It follows that \(C\) is weak Sidon with \(|C|=|A|+|B|\).
Moreover, any Sidon set \(S\subseteq C\) satisfies \(|S\cap A|\le h(A)\) and \(|S\cap(qB+t)|\le h(B)\),
and hence \(h(C)\le h(A)+h(B)\).
Therefore, we may assume \(|A|,|B|\ge 2\).

Consider the finite set
\[
\mathcal R:=\left\{\frac{a-a'}{b-b'}:\ a,a'\in A,\ b,b'\in B,\ b\neq b'\right\}\subset\mathbb{R}.
\]
Choose \(q>0\) with \(q\notin \mathcal R\), and then choose a large integer \(t\) so that
\begin{equation}\label{eq:t-large-wsidon}
t-D_A>\max\{D_A,qD_B\}.
\end{equation}
Define \(C:=A\cup(qB+t)\); we refer to either \( A \) or \( qB+t \) as a \emph{block}.
Condition~\eqref{eq:t-large-wsidon} easily implies that \(A\) lies strictly to the left of \(qB+t\),
so the union is disjoint and \(|C|=|A|+|B|\).

For later use, we call \(x+y\) a \emph{mixed sum} if \(x\in A\) and \(y\in qB+t\).
Since \(A\) lies strictly to the left of \(qB+t\), every mixed sum has the form
\[
a+(qb+t)=t+(a+qb)\qquad \mbox{ for some } a\in A \mbox{ and } b\in B.
\]
Since $\min A=\min B=0$, we have $\min\{a+qb:\ a\in A,\ b\in B\}=0$.
By~\eqref{eq:t-large-wsidon}, we have
$\max\{a+qb:\ a\in A,\ b\in B\}\le D_A+qD_B<t$ and
$\max\{a+a':a<a'\}\le 2D_A<t$.
Hence,
\[
t+\min\{a+qb:\ a\in A,\ b\in B\}
>
\max\{a+a':\ a<a',\ a,a'\in A\},
\]
and
\[
2t+\min\{q(b+b'):\ b<b',\ b,b'\in B\}
>
t+\max\{a+qb:\ a\in A,\ b\in B\}.
\]
This shows that every mixed sum is strictly larger than every sum of two distinct elements of \(A\), and strictly smaller than every sum
of two distinct elements of \(qB+t\). 
In particular, sums of distinct elements of $C$ coming from two different ``types'' of pairs cannot coincide.

It remains to rule out collisions between two mixed sums. Suppose
\[
a_1+(qb_1+t)=a_2+(qb_2+t).
\]
Then \(q(b_1-b_2)=a_2-a_1\). Since \(q\notin\mathcal R\), this forces \(b_1=b_2\) and \(a_1=a_2\).
Hence all mixed sums are pairwise distinct. Together with the fact that \(A\) and \(qB+t\) are weak Sidon,
it follows that all sums \(x+y\) with \(x<y\) in \(C\) are distinct, i.e.\ \(C\) is weak Sidon.

Finally, let \(S\subseteq C\) be Sidon. Then \(S\cap A\) and \(S\cap(qB+t)\) are Sidon subsets of the two blocks, and therefore
\[
|S|=|S\cap A|+|S\cap(qB+t)|
\le h(A)+h(qB+t)
= h(A)+h(B),
\]
where the last equality follows from Lemma~\ref{lem:wsidon_affine}. Taking the maximum over all Sidon subsets \(S\subseteq C\)
gives \(h(C)\le h(A)+h(B)\).
\end{proof}

With this lemma, one can promptly derive the following result.

\begin{proposition}\label{prop:g_subadditive}
For all $m,n\ge 1$ one has $g(m+n)\le g(m)+g(n)$.
\end{proposition}

\begin{proof}
Let $A,B$ be weak Sidon sets with $|A|=m$, $|B|=n$, $h(A)=g(m)$ and $h(B)=g(n)$.
Apply Lemma~\ref{lem:wsidon_concat} to obtain a weak Sidon set $C$ with $|C|=m+n$ and $h(C)\le h(A)+h(B)$. Since \(g(m+n)\) is the minimum of \(h(\cdot)\) over all weak Sidon sets of size \(m+n\), it follows that\[
g(m+n)\le h(C)\le g(m)+g(n).
\qedhere\]
\end{proof}

We can now complete the proof of the existence of the asymptotic constant $\gamma_*$.

\begin{proof}[Proof of the first assertion of Theorem~\ref{thm:limit_exists_v2}]
By Proposition~\ref{prop:g_subadditive}, the sequence $(g(n))_{n\ge 1}$ is subadditive. Lemma~\ref{lem:fekete} then implies that the limit $\lim_{n\to\infty} g(n)/n$ exists.
\end{proof}

\subsection{\texorpdfstring{$(4,5)$-sets: existence of the limit}{(4,5)-sets: existence of the limit}}

We now treat the \((4,5)\)-set setting using an approach analogous to that of the previous subsection.

\begin{lemma}\label{lem:affine}
Let \(A\subset \mathbb{R}\) be finite, \(t\in \mathbb{R}\), and \(q>0\). Set \(qA+t :=\{qa+t:\ a\in A\}\). Then \(A\) is a \((4,5)\)-set if and only if \(qA+t\) is a \((4,5)\)-set, and \(h(qA+t)=h(A)\).
\end{lemma}

\begin{proof}
For any $x,y\in A$ we have $|(qx+t)-(qy+t)|=q|x-y|$, so for any four points of \(A\), the six pairwise absolute differences are all rescaled by the same factor \(q\).
In particular, the number of distinct absolute differences is unchanged. Therefore \(A\) is a \((4,5)\)-set if and only if \(qA+t\) is a \((4,5)\)-set. 
By the same reasoning as in Lemma~\ref{lem:wsidon_affine}, we have \(h(qA+t)=h(A)\).
\end{proof}

In the next lemma,
we combine two \((4,5)\)-sets by placing them on well-separated scales so that mixed differences
cannot create new coincidences.
The proof is analogous to that of Lemma~\ref{lem:wsidon_concat}, with suitable modifications.

\begin{lemma}\label{lem:concat}
Let \(A,B\subset \mathbb{R}\) be two finite \((4,5)\)-sets. Then there exist \(q>0\) and \(t\in\mathbb{R}\)
such that the set
\[
C:=A \cup (qB+t)
\]
is a \((4,5)\)-set, has \(|C|=|A|+|B|\), and satisfies \(h(C)\le h(A)+h(B)\).
\end{lemma}

\begin{proof}
By Lemma~\ref{lem:affine}, we may translate \(A\) and \(B\) so that \(\min A=\min B=0\).
Write \(D_A:=\max A\) and \(D_B:=\max B\).
Consider the finite set
\[
\mathcal R:=\left\{\frac{a-a'}{b-b'}:\ a,a'\in A,\ b,b'\in B,\ b\neq b'\right\}\subset\mathbb{R}.
\]
Choose \(q>0\) with \(q\notin \mathcal R\), and then choose \(t\) large enough so that
\begin{equation}\label{eq:t-large}
t-D_A>\max\{D_A,qD_B\}.
\end{equation}
Define \(C:=A\cup(qB+t)\).
Similarly as before, we refer to either subset \( A \) or \( qB+t \) as a \emph{block}.
Condition~\eqref{eq:t-large} easily implies that \(A\) lies strictly to the left of \(qB+t\),
so the union is disjoint and \(|C|=|A|+|B|\).

For later use, we call \(|x-y|\) a \emph{mixed difference} if \(x\in A\) and \(y\in qB+t\).
Since \(A\) lies strictly to the left of \(qB+t\), every mixed difference is positive and has the form \((qb+t)-a\) with \(a\in A\) and \(b\in B\). By~\eqref{eq:t-large},
\[
(qb+t)-a \ge t-D_A > \max\{D_A,qD_B\}.
\]
On the other hand, every difference inside \(A\) has absolute value at most \(D_A\), and every difference inside \(qB+t\)
has absolute value at most \(qD_B\). It follows that no mixed difference can coincide with a within-block difference.

Moreover, if two mixed differences coincide, say
\[
qb_1+t-a_1=qb_2+t-a_2,
\]
then \(q(b_1-b_2)=a_1-a_2\). Since \(q\notin\mathcal R\), this forces \(b_1=b_2\) and \(a_1=a_2\).
Thus all mixed differences are pairwise distinct.

We now verify that \(C\) is a \((4,5)\)-set. Let \(X\subseteq C\) with \(|X|=4\).
If \(X\) is contained in one block, then \(X\) determines at least five distinct pairwise absolute differences,
because this holds for \(A\) and for \(B\), and the \((4,5)\)-property is preserved by Lemma~\ref{lem:affine}.

Assume therefore that \(X\) meets both blocks. If \(X\) satisfies $|X\cap A|\in \{1,3\}$, then it contributes three mixed differences,
which are pairwise distinct, and the triple in one block contributes at least two distinct
within-block differences. Since mixed differences cannot coincide with within-block differences, \(X\) determines at least
\(3+2=5\) distinct absolute differences. 
Otherwise, we must have \(|X\cap A|=2\). 
Then it contributes four mixed differences, again pairwise distinct, together with
one within-block difference from each side. Even if these two within-block differences coincide, we still obtain at least
\(4+1=5\) distinct absolute differences. Hence every \(4\)-subset of \(C\) determines at least five distinct pairwise
absolute differences, so \(C\) is a \((4,5)\)-set.

Finally, let \(S\subseteq C\) be Sidon. Then \(S\cap A\) and \(S\cap(qB+t)\) are Sidon subsets of the two blocks, and therefore
\[
|S|=|S\cap A|+|S\cap(qB+t)|
\le h(A)+h(qB+t)
= h(A)+h(B),
\]
where the last equality follows from Lemma~\ref{lem:affine}. Taking the maximum over all Sidon subsets \(S\subseteq C\)
gives \(h(C)\le h(A)+h(B)\).
\end{proof}

We are now ready to deduce the key subadditivity property of the extremal function \(f(n)\).
\begin{proposition}\label{prop:subadd}
For all $m,n\ge 1$ one has $f(m+n)\le f(m)+f(n)$.
\end{proposition}

\begin{proof}
Fix \(m,n\ge 1\). Choose a \((4,5)\)-set \(A\) with $|A|=m$ and $h(A)=f(m)$. Similarly, choose a \((4,5)\)-set \(B\) with \(|B|=n\) and \(h(B)=f(n)\). By Lemma~\ref{lem:concat}, there exist \(q>0\) and \(t\in\mathbb R\) such that \(C=A\cup(qB+t)\) is a \((4,5)\)-set with \(|C|=m+n\) and \(h(C)\le h(A)+h(B)=f(m)+f(n)\). 
By the definition of $f(m+n)$, we have
\[
f(m+n)\le h(C)\le f(m)+f(n).
\qedhere\]
\end{proof}

We can now complete the proof of the existence of the limit $\lim_{n\to\infty} f(n)/n$.

\begin{proof}[Proof of Theorem~\ref{thm:limit_exists_v1}]
By Proposition~\ref{prop:subadd}, the sequence \((f(n))_{n\ge 1}\) is subadditive.
Hence, Lemma~\ref{lem:fekete} implies that the limit \(\lim_{n\to\infty} f(n)/n\) exists and equals \(\inf_{n\ge 1} f(n)/n\), which is precisely the optimal constant \(c_*\) in Problem~\ref{prob:erdos-757}.
\end{proof}

\section{Sidon subsets in weak Sidon sets}\label{sec:weak_sidon}

In Section~\ref{sec:existence_of_limit} we established the first assertion of Theorem~\ref{thm:limit_exists_v2}, so the
constant $\gamma_*$ is well defined. In this section we prove Theorem~\ref{thm:gamma_star_bounds}
by establishing matching upper and lower bounds.
In particular, this will imply the precise value of \( \gamma_*\) and thus prove the second assertion of Theorem~\ref{thm:limit_exists_v2}.

\subsection{\texorpdfstring{Upper bound on $g(n)$}{Upper bound on g(n)}}

We begin with the explicit definition of the weak Sidon sets required for our construction.
For an integer $n\ge 2$, define
\[
X_i:=2^{\,n-i}3^{\,i} \quad \mbox{for} \quad 0\le i\le n-1,
\qquad 
Y_i:=2^{\,n-i+1}3^{\,i} \quad \mbox{for} \quad 0\le i\le n,
\]
and set
\[
A_{2n+1}:=\{X_0,\dots,X_{n-1}\} \cup \{Y_0,\dots,Y_n\}.
\]
Clearly, we have $|A_{2n+1}|=2n+1$.

First, we need to show that the sets $A_{2n+1}$ are indeed weak Sidon. 
To proceed, we establish the following two lemmas, 
which say that sums in $A_{2n+1}$ are essentially unique.
For $m\in\mathbb{Z}_{>0}$, let $v_3(m)$ denote the largest integer $k\ge 0$ such that $3^k\mid m$.

\begin{lemma}\label{lem:wsidon_recover_minv3}
Let $a,b\in A_{2n+1}$ be distinct, and write $S=a+b$.
Assume that $v_3(a)\neq v_3(b)$. Then the unordered pair $\{a,b\}$ is uniquely determined by $S$.
\end{lemma}

\begin{proof}
Let $i:=\min\{v_3(a),v_3(b)\}$. Since $v_3(a)\neq v_3(b)$, we have
\[
v_3(S)=v_3(a+b)=\min\{v_3(a),v_3(b)\}=i.
\]
Thus $i$ is uniquely determined by $S$.

Assume without loss of generality that $v_3(a)=i<v_3(b)$. Then $b\equiv 0\pmod{3^{i+1}}$, whereas
$a\not\equiv 0\pmod{3^{i+1}}$. Hence
\[
S=a+b\equiv a \pmod{3^{i+1}}.
\]
In $A_{2n+1}$ there are exactly two elements of $3$-adic valuation $i$, namely $X_i$ and $Y_i$, and
both are nonzero modulo $3^{i+1}$. Moreover,
\[
X_i\equiv 2^{\,n-i}\cdot 3^i \pmod{3^{i+1}},
\qquad
Y_i\equiv 2^{\,n-i+1}\cdot 3^i \pmod{3^{i+1}},
\]
so $X_i\not\equiv Y_i\pmod{3^{i+1}}$. Since $a\in\{X_i,Y_i\}$ and $X_i\not\equiv Y_i\pmod{3^{i+1}}$, the congruence
$S\equiv a\pmod{3^{i+1}}$ determines $a$ uniquely. Consequently $b=S-a$ is uniquely determined,
and hence the unordered pair $\{a,b\}$ is uniquely determined by $S$.
\end{proof}

\begin{lemma}\label{lem:wsidon_recover_equalv3}
Let $a,b\in A_{2n+1}$ be distinct, and write $S=a+b$.
Assume that $v_3(a)=v_3(b)$. 
Then the unordered pair \(\{a,b\}\) is uniquely determined by \(S\); moreover, there do not exist distinct \(c,d\in A_{2n+1}\) with \(v_3(c)\neq v_3(d)\) such that \(S=c+d\).
\end{lemma}

\begin{proof}
Let $i:=v_3(a)=v_3(b)$. Since $A_{2n+1}$ contains exactly two elements of $3$-adic valuation $i$ for $0\le i\le n-1$, namely $X_i$ and $Y_i$,
while the valuation $n$ occurs only for $Y_n$, the assumption $a\ne b$ forces $i\le n-1$ and $\{a,b\}=\{X_i,Y_i\}$. Hence
\[
S=a+b=X_i+Y_i = 2^{\,n-i}3^i + 2^{\,n-i+1}3^i = 2^{\,n-i}3^{i+1}.
\]
In particular, $v_3(S)=i+1$, so $i$ is uniquely determined by $S$ as $i=v_3(S)-1$.
Thus any representation of $S$ as a sum of two distinct elements of $A_{2n+1}$ with equal
$3$-adic valuation must be the pair $\{X_i,Y_i\}$.
This proves the first assertion of this lemma. 

It remains to rule out a representation $S=c+d$ with $c,d\in A_{2n+1}$ distinct and
$v_3(c)\neq v_3(d)$. In that case one has
\[
v_3(S)=\min\{v_3(c),v_3(d)\}=i+1,
\]
so the summand of smaller $3$-adic valuation has valuation $i+1$.
If $i=n-1$, then this forces that summand to be $Y_n=S$, so the other summand equals $0$, contradicting
$0\notin A_{2n+1}$. Hence we may assume $i\le n-2$, so the smaller-valuation summand is either $X_{i+1}$ or $Y_{i+1}$.

If it is $Y_{i+1}$, then $Y_{i+1}=2^{\,n-i}3^{i+1}=S$, forcing the other summand to be $0$,
again contradicting $0\notin A_{2n+1}$.
If it is $X_{i+1}$, then $X_{i+1}=2^{\,n-i-1}3^{i+1}=S/2$, forcing the other summand
to equal $S-X_{i+1}=X_{i+1}$, contradicting distinctness. Hence no such representation exists.
\end{proof}

Now we can verify that the sets $A_{2n+1}$ are indeed weak Sidon. 

\begin{proposition}\label{prop:A2n1_weak_sidon}
For every $n\ge 2$, the set $A_{2n+1}$ is a weak Sidon set.
\end{proposition}

\begin{proof}
In this proof, we call a pair \(\{a,b\}\subset A_{2n+1}\) type~1 if \(v_3(a)\neq v_3(b)\) and type~2 otherwise.
For a contradiction, let \(a<b\) and \(c<d\) be two pairs of distinct elements of \(A_{2n+1}\) such that
\(S:=a+b=c+d\).
By Lemma~\ref{lem:wsidon_recover_minv3} and the first assertion of Lemma~\ref{lem:wsidon_recover_equalv3},
one of the pairs (say \(\{a,b\}\)) must be of type~1 and the other pair \(\{c,d\}\) of type~2.
However, this contradicts the second assertion of Lemma~\ref{lem:wsidon_recover_equalv3}.
Hence all sums $x+y$ with $x<y$ in
$A_{2n+1}$ are distinct, i.e.\ $A_{2n+1}$ is weak Sidon.
\end{proof}

We require further properties of \(A_{2n+1}\) and its A.P.-hypergraph \(H(A_{2n+1})\).
The next lemma characterizes all \(3\)-term arithmetic progressions in \(A_{2n+1}\).

\begin{lemma}\label{lem:A2n1_all_3AP}
All $3$-term arithmetic progressions in $A_{2n+1}$ are given by the following $2n-1$ triples:
\[
\{X_i,X_{i+1},Y_i\}\quad \mbox{for} \quad 0\le i\le n-2,
\quad \mbox{and} \quad
\{X_i,Y_i,Y_{i+1}\}\quad \mbox{for} \quad 0\le i\le n-1.
\]
\end{lemma}

\begin{proof}
Direct computation shows
\[
X_i+Y_i = 2^{\,n-i}3^i + 2^{\,n-i+1}3^i = 2\cdot 2^{\,n-i-1}3^{i+1} = 2X_{i+1}
\quad \mbox{for} \quad 0\le i\le n-2,
\]
and
\[
X_i+Y_{i+1}
= 2^{\,n-i}3^i + 2^{\,n-i}3^{i+1}
= 4\cdot 2^{\,n-i}3^i
= 2Y_i
\quad \mbox{for} \quad 0\le i\le n-1.
\]
Thus the listed $2n-1$ triples are indeed $3$-term arithmetic progressions in $A_{2n+1}$.

Since \(A_{2n+1}\) is weak Sidon (by Proposition~\ref{prop:A2n1_weak_sidon}),
Lemma~\ref{lem:mleqn} implies that the midpoint map on its A.P.-hypergraph
\(H(A_{2n+1})\) is injective. Consequently,
\[
m_{H(A_{2n+1})} \le n_{H(A_{2n+1})}-2 = |A_{2n+1}|-2 = 2n-1.
\]
Since we have already listed \(2n-1\) distinct \(3\)-term arithmetic progressions, no others exist.
\end{proof}

We need one more result before proving the main theorem of this subsection, which determines the exact value of $h(A_{2n+1})$.

\begin{theorem}\label{thm:A2n1_h_equals_nplus1}
For every $n\ge 2$, the set $A_{2n+1}$ satisfies \(h(A_{2n+1}) = n+1\).
\end{theorem}

\begin{proof}
By Proposition~\ref{prop:A2n1_weak_sidon}, $A_{2n+1}$ is weak Sidon. 
Let $H_{2n+1}=H(A_{2n+1})$ be its A.P.-hypergraph. 
Lemma~\ref{lem:sidon-iff-apfree} gives $h(A_{2n+1})=\alpha(H_{2n+1})$; hence it suffices to show $\alpha(H_{2n+1})=n+1$.

First, we prove the lower bound $\alpha(H_{2n+1})\geq n+1$.
We claim that the set $\{Y_0,\dots,Y_n\}$ is an independent set in $H_{2n+1}$ (i.e., $3$AP-free). Indeed, by Lemma~\ref{lem:A2n1_all_3AP} we observe that every $3$-term arithmetic progression in $A_{2n+1}$ contains some $X_i$. Hence we have $\alpha(H_{2n+1})\ge n+1$.

To show the upper bound $\alpha(H_{2n+1})\leq n+1$,
consider any $3$AP-free $T\subseteq A_{2n+1}$. Write
\[
U:=\{\,i\in\{0,\dots,n-1\}: X_i\in T\,\}.
\]
Decompose $U$ into disjoint maximal intervals of consecutive integers:
\[
U=\bigsqcup_{j=1}^r ([s_j,t_j]\cap\mathbb{Z}),
\quad \mbox{where} \quad s_j\le t_j \quad \mbox{and} \quad t_j+1<s_{j+1}.
\]
Fix one such interval $[s,t]$ and set $\ell:=t-s+1$.

For each $i$ with $s\le i<i+1\le t$ we have $i,i+1\in U$, and hence $X_i,X_{i+1}\in T$. By Lemma~\ref{lem:A2n1_all_3AP}, the triple $\{X_i,X_{i+1},Y_i\}$ is a $3$AP, and thus we must have $Y_i\notin T$.
Hence
\[
Y_s,Y_{s+1},\dots,Y_{t-1}\notin T.
\]
Moreover, the triple $\{X_t,Y_t,Y_{t+1}\}$ is a $3$AP, so since $X_t\in T$ we cannot have
$Y_t,Y_{t+1}\in T$ simultaneously. Therefore at least one of $Y_t$ and $Y_{t+1}$ is not in $T$.
Consequently, the interval $[s,t]$ forces at least $\ell$ distinct $Y$-vertices to be absent from $T$
(namely, the $\ell-1$ vertices $Y_s,\dots,Y_{t-1}$ together with one of $Y_t$ or $Y_{t+1}$).

Because the intervals are separated by at least one index, these forced exclusions coming from different
intervals are disjoint. Hence $U$ forces at least $|U|$ vertices from $\{Y_0,\dots,Y_n\}$ to be absent
from $T$. It follows that
\(
|T\cap\{Y_0,\dots,Y_n\}|\le (n+1)-|U|.
\)
Therefore
\[
|T|=|U|+|T\cap\{Y_0,\dots,Y_n\}|\le |U|+\bigl((n+1)-|U|\bigr)=n+1.
\]
Thus $\alpha(H_{2n+1})\le n+1$. 
This completes the proof that $h(A_{2n+1})=\alpha(H_{2n+1})=n+1$.
\end{proof}

With the preceding results in hand, we can derive the upper bound \(g(n)\le \left\lceil \frac{n+1}{2}\right\rceil\) as follows.

\begin{proof}[Proof of the upper bound in Theorem~\ref{thm:gamma_star_bounds}]
We prove that for every $n\ge1$,
\[
g(n)\le \left\lceil \frac{n+1}{2}\right\rceil.
\]For odd $n=2k+1$ with $k\ge2$, Proposition~\ref{prop:A2n1_weak_sidon} and
Theorem~\ref{thm:A2n1_h_equals_nplus1} give that $A_{2k+1}$ is weak Sidon and
$h(A_{2k+1})=k+1$. Hence
\[
g(2k+1)\le h(A_{2k+1})=k+1=\left\lceil \frac{(2k+1)+1}{2}\right\rceil.
\]For even $n=2k$ with $k\ge2$, let $B$ be any $(2k)$-element subset of $A_{2k+1}$
(obtained by deleting one element). Since the weak Sidon property is inherited by
subsets, $B$ is weak Sidon. Moreover, any Sidon subset of $B$ is also a Sidon
subset of $A_{2k+1}$, so $h(B)\le h(A_{2k+1})=k+1$. Therefore,
\[
g(2k)\le h(B)\le k+1=\left\lceil \frac{(2k)+1}{2}\right\rceil.
\]
Finally, the remaining cases $n\in\{1,2,3\}$ are immediate, and thus the stated
upper bound holds for all $n\ge1$.
\end{proof}

\subsection{\texorpdfstring{Lower bound on $g(n)$}{Lower bound on g(n)}}
To show the desired lower bound on $\gamma_*$, we need the following upper bound on the transversal number of a $3$-uniform hypergraph, due to Chv\'{a}tal and McDiarmid~\cite{ChMc92} and independently Tuza~\cite{Tuz90} (see also~\cite[Theorem~1]{HeYe21}).

\begin{theorem}[\cite{ChMc92,Tuz90}]\label{thm:CMT_v1}
Let $H$ be a $3$-uniform hypergraph. Then \(4\tau(H)\le n_H+m_H\).
\end{theorem}

We now prove the lower bound in Theorem~\ref{thm:gamma_star_bounds}, namely \(g(n)\ge \left\lceil \frac{n+1}{2}\right\rceil\) for all $n\ge 1$.

\begin{proof}[Proof of the lower bound in Theorem~\ref{thm:gamma_star_bounds}]
Let $A\subset \mathbb{R}$ be any weak Sidon set of size $n$, and let $H=H(A)$ be its A.P.-hypergraph.
If $n=1$, then $h(A)=1=\frac12 n+\frac12$. 
So we may assume $n\ge 2$.
Then the following hold:
\begin{enumerate}
    \item $H$ is $3$-uniform;
    \item $m_H\le n_H-2$ (by Lemma~\ref{lem:mleqn});
    \item $h(A)=\alpha(H)$ (by Lemma~\ref{lem:sidon-iff-apfree}).
\end{enumerate}
Then, using Theorem~\ref{thm:CMT_v1}, we can derive that
\[
\tau(H)\le \frac14(n_H+m_H)\le \frac14\bigl(n_H+(n_H-2)\bigr)=\frac12n_H-\frac12.
\]
Hence
\[
h(A)=\alpha(H)=n_H-\tau(H)\ge n_H-\Bigl(\frac12n_H-\frac12\Bigr)=\frac12n_H+\frac12=\frac12n+\frac12.
\]
Therefore, we have
\[
g(n)=\min_{A} h(A) \ge \frac12 n+\frac12
\qquad\text{for all }n\ge 1.
\]
Since $g(n)$ is an integer, this implies
\[
g(n)\ \ge\ \left\lceil \frac{n+1}{2}\right\rceil,
\]
as desired.
\end{proof}

Together with the matching upper bound proved earlier, we obtain
\[
g(n)=\left\lceil \frac{n+1}{2}\right\rceil
\qquad\text{for all } n\ge 1.
\]
This establishes Theorem~\ref{thm:gamma_star_bounds}. In particular, \(\gamma_*=\frac12\). Combined with the existence statement established earlier, this completes the proof of Theorem~\ref{thm:limit_exists_v2}.

\section{\texorpdfstring{Sidon subsets in $(4,5)$-sets}{Sidon subsets in (4,5)-sets}}\label{sec:c_star_bounds}

In this section we prove Theorem~\ref{thm:c_star_bounds} by improving both the upper and lower bounds on \(c_*\).

\subsection{\texorpdfstring{An improved upper bound on $c_*$}{An improved upper bound on c*}}\label{subsec:A_base_v1}

We begin with a concrete $14$-point $(4,5)$-set whose largest Sidon subset has size $8$. Let
\[
A_{\mathrm{base}}
=\{0,136,200,243,246,249,272,286,298,323,400,528,596,1056\}\subset\mathbb Z.
\]

\begin{lemma}\label{lem:base-block}
The set \(A_{\mathrm{base}}\) is a \((4,5)\)-set, and \(h(A_{\mathrm{base}})=8\).
\end{lemma}

\begin{proof}
Both assertions were verified by an exhaustive computer search using exact integer arithmetic;
see the code repository cited in~\cite{Git}.

\medskip

\noindent\underline{\textit{Verification of the $(4,5)$-property.}}
The program enumerates all $\binom{14}{4}=1001$ four-element subsets $B\subseteq A_{\mathrm{base}}$.
For each such $B=\{x_1,x_2,x_3,x_4\}$ it computes the set
\[
\{\,|x_i-x_j|:\ 1\le i<j\le 4\,\}
\]
of the six pairwise absolute differences and checks that its cardinality is at least $5$.
Since no counterexample is found, $A_{\mathrm{base}}$ is a $(4,5)$-set.

\medskip

\noindent\underline{\textit{Computation of $h(A_{\mathrm{base}})$.}}
To determine the largest size of a Sidon subset, the program enumerates all $2^{14}=16384$ subsets
$S\subseteq A_{\mathrm{base}}$ and tests whether $S$ is Sidon by checking that all sums $x+y$
with $x\le y$ and $x,y\in S$ are pairwise distinct.
The maximum size encountered is $8$, and the script outputs the explicit Sidon subset
\[
\{0,136,200,243,246,298,323,528\}.
\]
Thus \(h(A_{\mathrm{base}})\ge 8\), while the exhaustive search shows that no larger Sidon subset exists.
Therefore \(h(A_{\mathrm{base}})=8\).
\end{proof}

With this example in hand, the upper bound \(c_*\le 4/7\) follows immediately from Theorem~\ref{thm:limit_exists_v1}.

\begin{proof}[Proof of the upper bound in Theorem~\ref{thm:c_star_bounds}]
By Theorem~\ref{thm:limit_exists_v1} we have \(c_{*}=\inf_{n\ge 1}\frac{f(n)}{n}\).
Taking \(n=14\) and using Lemma~\ref{lem:base-block}, we obtain \(f(14)\le h(A_{\mathrm{base}})=8\). Therefore
\[
c_{*}=\inf_{n\ge 1}\frac{f(n)}{n} \le \frac{f(14)}{14}\le \frac{8}{14}=\frac{4}{7},
\]
as claimed.
\end{proof}

\subsection{\texorpdfstring{An improved lower bound on $c_*$}{An improved lower bound on c*}}

In the lower-bound direction, we follow the A.P.-hypergraph approach of Gy\'arf\'as and Lehel~\cite{GyLe95},
but we strengthen the transversal estimate using a result of Henning and Yeo~\cite{HeYe21}.
This result, well suited to our setting, does not appear to have been noted previously.

For a \((4,5)\)-set $A$, Gy\'arf\'as and Lehel~\cite[Proposition~2.6]{GyLe95} showed that \(H(A)\) avoids the \(7\)-vertex configuration \(F_7\) shown in Figure~\ref{fig:F7}.

\begin{lemma}[\cite{GyLe95}]\label{lem:H0free}
If $A$ is a $(4,5)$-set, then $H(A)$ contains no subhypergraph isomorphic to $F_7$.
\end{lemma}

We invoke a transversal bound for $3$-uniform linear $F_7$-free hypergraphs due to Henning and Yeo~\cite[Theorem~7]{HeYe21}.

\begin{theorem}[\cite{HeYe21}]\label{thm:HY}
Let $H$ be a $3$-uniform linear hypergraph that contains no subhypergraph isomorphic to $F_7$.
Then \(17\tau(H) \le 5 n_H + 3 m_H\).
\end{theorem}

We are now ready to prove the improved lower bound for \(c_*\).

\begin{proof}[Proof of the lower bound in Theorem~\ref{thm:c_star_bounds}]
Let $A$ be a $(4,5)$-set of size $n$ and let $H=H(A)$. Since every \((4,5)\)-set is weak Sidon by Proposition~\ref{prop:45_implies_weak}, Lemma~\ref{lem:sidon-iff-apfree} gives \(h(A)=\alpha(H)\). By definition, \(H\) is \(3\)-uniform; by Lemmas~\ref{lem:linear} and \ref{lem:H0free}, it is linear and \(F_7\)-free. Hence Theorem~\ref{thm:HY} applies and gives
\[
17\tau(H)\le 5n_H+3m_H.
\]
Using Lemma~\ref{lem:mleqn}, we have \(m_H\le n_H-2\) (for \(|A|\ge 2\)), and therefore
\[
17\tau(H)\le 5n_H+3(n_H-2)\le 8n_H.
\]
Thus \(\tau(H)\le \frac{8}{17}n_H\). Since \(\alpha(H)+\tau(H)=n_H\) for every hypergraph, it follows that
\[
h(A)=\alpha(H)\ge n_H-\frac{8}{17}n_H=\frac{9}{17}n_H=\frac{9}{17}n.
\]
Therefore \(h(A)\ge \frac{9}{17}n\) for every \((4,5)\)-set \(A\) of size \(n\ge 2\), and taking the minimum over all such
\(A\) yields \(f(n)\ge \frac{9}{17}n\) for all \(n\ge 2\). For \(n=1\), we trivially have \(f(1)=1\ge \frac{9}{17}\).
Hence \(f(n)\ge \frac{9}{17}n\) for all \(n\ge 1\). Finally, using \(c_*=\inf_{n\ge 1}\frac{f(n)}{n}\), we conclude that \(c_*\ge \frac{9}{17}\).
\end{proof}

\bigskip

{\noindent \bf Acknowledgements.}
We are grateful to Thomas Bloom for founding and maintaining the \emph{Erd\H{o}s Problems} website~\cite{EP757}, and to Nat Sothanaphan for carefully reading an earlier version of this manuscript. The first author was supported by National Key Research and Development Program of China 2023YFA1010201, National Natural Science Foundation of China grant 12125106, and Innovation Program for Quantum Science and Technology 2021ZD0302902.

\medskip

{\noindent \bf Declaration of AI usage.}
An AI assistant was used in the search for a construction of the set \(A_{\mathrm{base}}\) in Lemma~\ref{lem:base-block}. 
This AI-assisted exploration first produced a \(12\)-point example, yielding the improved upper bound \(c_*\le 7/12\), and subsequently, after several rounds of refinement, the \(14\)-point set \(A_{\mathrm{base}}\) used in the paper, which yields the reported upper bound \(c_*\le 4/7\). 
The AI assistant also helped generate an initial draft of the code used in the exhaustive verification for Lemma~\ref{lem:base-block}. 
Except for the work related to Lemma~\ref{lem:base-block}, all proofs in this paper and their presentation were carried out by the human authors.

\end{document}